\documentclass[11pt,a4paper,reqno]{amsart}
\usepackage[pdfstartview=FitH]{hyperref} 
\usepackage{amssymb,xcolor}
\usepackage[top=2.3cm,right=2.5cm,left=2.5cm,bottom=2.3cm]{geometry}
\usepackage[english]{babel}
\usepackage{pdflscape,url,xcolor,graphicx}
\usepackage[T1]{fontenc}
\usepackage{palatino,eulervm}
\usepackage[scaled=1.03]{inconsolata}
\usepackage{pgfplots}
\usepackage{pgfplotstable}
\pgfplotsset{width=15cm}
\usepackage{tkz-euclide}
\usepackage{tikz}
\usepackage{verbatim}
\usetikzlibrary{shadows}
\usetkzobj{all}

\allowdisplaybreaks[4]

 \newcommand{\field}[1]{\mathbb{#1}}

\newcommand{\Z}{\field{Z}}
\newcommand{\N}{\field{N}}
\newcommand{\Mypm}{\mathbin{\tikz [x=1.4ex,y=1.4ex,line width=.1ex] \draw (0.0,0) -- (1.0,0) (0.5,0.08) -- (0.5,0.92) (0.0,0.5) -- (1.0,0.5);}}%

\renewcommand{\vec}{\boldsymbol}

\newtheorem{thm}{Theorem}[section]
\newtheorem{lemma}[thm]{Lemma}
\newtheorem{prop}[thm]{Proposition}

\newtheorem{cor}[thm]{Corollary}

\newenvironment{remark}[1][Remark]{\begin{trivlist}
\item[\hskip \labelsep {\bfseries #1}]}{\end{trivlist}}

\title[Hyperbinary Expansions]{On a Graph Connecting Hyperbinary Expansions}
\author{Maurizio Brunetti}
\author{Alma D'Aniello}
\begin{document}
\address{\, \newline
Dipartimento di Matematica e applicazioni `R. Caccioppoli',\newline
Universit\`a di Napoli `Federico II',\newline
 Piazzale Tecchio 80,   I-80125 Napoli, 
Italy. \newline \newline
 E-mail: mbrunett@unina.it, daniello@unina.it}
 \subjclass[2010]{05C30, 11B75, 68Q42, 68R15}
\renewcommand{\abstractname}{Abstract}
\begin{abstract}
Le $n$ be any positive integer. A hyperbinary expansion of $n$ is a representation of $n$ as sum of powers of $2$, each power being used at most twice. In this paper we study some properties of a suitable edge-coloured and vertex-weighted oriented graph $A(n)$ whose nodes are precisely the several hyperbinary representations of $n$. In particular, we identify those integers $m \in \N$ such that the fundamental group of $A(m)$ is abelian.
\end{abstract}

\maketitle
\section{Introduction}
A hyperbinary expansion of the positive integer $n$ is a word 
$x_1 \dots x_k$ with
$x_i \in \{ 0,1,2\}$, $x_1\not=0$, and
$n = \sum_{i=1}^k x_i 2^{k-i}$.

The last decades have seen a growing interest toward hyperbinary expansions, especially since N.~Calkin and  H.~S.~Wilf proved in \cite{CW} that all positive rationals appear just once in the sequence
$$ \Big\{ \frac{b(n)}{b(n+1)} \Big\}_{n\geq 0},$$ 
where $b(0)=1$, and $b(n)$ for $n >0$ is the number of the hyperbinary expansions of $n$.  In any case, many intriguing properties of the function
$ b : \N_0 \rightarrow \N $ had been already examined in \cite{R}, where $b(n)$ is called {\em the $3$-rd binary partition function} and denoted by $b(3;n)$.

More recently, A.~De~Luca and C.~Reutenauer explained in \cite{AR} how hyperbinary expansions are related to the nodes of the Christoffel tree, first introduced in \cite{BdL}. For their part, hyper-$m$-ary expansions for $m \geq 2$ have been considered in \cite{M}.

Let  $\mathcal H(n)$ be the set of all hyperbinary expansions of a fixed $n \in \N$.
As explained in Section~\ref{s2},  such set contains a unique expansion $\vec{n}'$ not containing $0$'s.  We refer to $\vec{n}'$ as the {\em minimal hyperbinary expansion} of $n$, since any other element in $\mathcal H(n)$ is bigger than $\vec{n}'$ with respect to the so-called shortlex ordering (see Corollary~\ref{conf} below). 

On the other hand, the maximal expansion $\vec{n}'' \in \mathcal H(n)$ with respect to shortlex ordering turns out to be the unique binary expansion of $n$. 

The directed graph $A(n) = (\mathcal H(n), \mathcal E(n))$  we are going to introduce in Section~\ref{s3} has $\vec{n}'$  as root, and $\vec{n}''$ as unique terminal node. This  is one more reason  to consider $\vec{n}'$  and $\vec{n}''$ as the {\em extrema} of  $\mathcal H(n)$.

Among its features, the graph $A(n)$ induces a new partial order on  $\mathcal H(n)$ related to 
the number of ancestors of a fixed vertex $\vec{n}$ in  $A(n)$, giving an alternative significant method to measure how far $\vec{n}$  is from being binary or minimal hyperbinary.

The graph  $A(n)$ is also useful to visualize some properties of expansions of longest length. We discuss this topic in Section \ref{s4}.

The final two sections of this paper are devoted to identify those integers giving rise to graphs with small cyclomatic numbers.

\section{Basic Properties and Tools}\label{s2}
To introduce our basic tools, we mix up some standard terminology and notations borrowed from the theory of formal languages and the theory of directed graphs. Our main sources for them are \cite{BKR} and  \cite{BO}.

Let $\Sigma^*$ be the free monoid over the {\em alphabet} $\Sigma = \{ 0,1,2 \}$. The elements of $\Sigma^*$ are called strings or words. The trivial element in $\Sigma^*$ is the {\em empty string}. Once we introduce the equivalence relation $\sim$ that identifies two words in $\Sigma^*$ differing only in zeros in the left-hand side, each hyperbinary expansion in  $\mathcal H = \cup_{n>0} \mathcal H(n)$ can be regarded as a suitable equivalence class in $\Sigma^*/\sim$. From such perspective the three words $00210$, $0210$ and $210$ in $\Sigma^*$   all represent the same element in $\mathcal H (10)$.

We now consider a {\em string-rewriting system} 
$$\vec{R} = \{ (02, \,  10), \; (12, \, 20)  \} \subset \Sigma^* \times \Sigma^*.$$
We call its elements {\em rewrite rules}. An alternative way to denote them is
$ 0  2 \rightarrow 1 0$ and $1  2 \twoheadrightarrow 2  0$.
We also use the same type of arrows to denote any other {\em single-step reduction} induced by $\vec{R}$, i.e.\ 
\begin{equation}\label{rew} \vec{x} \,0 \, 2 \, \vec{y} \rightarrow \vec{x} \,1 \, 0 \, \vec{y}  \qquad \text{and} \qquad \vec{x} \,1 \, 2 \, \vec{y} \twoheadrightarrow \vec{x} \,2 \, 0 \, \vec{y} \qquad \forall \; \vec{x}, \vec{y} \in \Sigma^*. 
\end{equation}
Suppose now $\vec{u}$ and $\vec{v}$ in $\Sigma^*$ connected by a finite number $k>0$ of single-step reductions. In this case, we write 
\begin{equation}\label{rel}  \vec{u}  \overset{*}{\rightarrow}_{\vec{R}} \vec{v}. \end{equation} 
The word  $\vec{u}$ is called an {\em ancestor of $\vec{v}$}, and $\vec{v}$ is {\em a descendant of  $\vec{u}$}.
Furthermore, if $k=1$, we also say that $\vec{u}$ is a {\em parent of $\vec{v}$}, and $\vec{v}$ a {\em child of  $\vec{u}$}.

We point out that in \cite{BKR} and \cite{BO} the authors  denote by $\overset{*}{\rightarrow}_{\vec{R}} $ the reflexive, transitive closure of $\vec{R}$; on the contrary, the binary relation \eqref{rel} denoted here by the same symbol is just the transitive closure of $\vec{R}$, but it is not reflexive. 

Note that $\vec{R}$ naturally defines a rewriting system  on $\mathcal H = \cup_{n>0} \mathcal H(n)$ (denoted by $\vec{R}_{\mathcal H}$), with the caveat that the hyperbinary expansion $2 \, x_1\cdots x_k$, also represented by $0 \, 2 \, x_1\cdots x_k$ in $\Sigma^*$, is a parent of $1 \, 0 \, x_1\cdots x_k$ with respect to $\vec{R}_{\mathcal H}$.

\begin{prop}\label{p4} Let $\vec{n}$ be any element in $\mathcal H(n)$. All ancestors and descendants of $\vec{n}$ with respect to $\vec{R}_{\mathcal H}$  belong to  
 $\mathcal H(n)$.
\end{prop}
\begin{proof}
The statement immediately comes from the two trivial arithmetic identities
$$ 0 \cdot 2^{s}+ 2 \cdot 2^{s-1} = 1 \cdot 2^{s} + 0 \cdot 2^{s-1} \quad \text{and} \quad 1 \cdot 2^s + 2 \cdot 2^{s-1} = 2 \cdot 2^s + 0 \cdot 2^{s-1} \qquad \forall \; s \in \N.
$$
\end{proof}

We now recall a renowned result concerning the function
\begin{equation}
b : n \in \N_ 0 \longmapsto \begin{cases} 0 \quad  \qquad \;  \text{if $n=0$;}\\
\lvert \mathcal H(n) \rvert \quad \text{otherwise.}
\end{cases}
\end{equation}
\begin{prop}\label{pr1}
The sequence $(b(n))_{n \geq 0}$ may be defined recursively by
\begin{equation}
b(0)=b(1)=1; \quad b(2n+1)=b(n); \quad b(2n)= b(n)+b(n-1).
\end{equation}
\end{prop}
\begin{proof} See \cite[Corollary~2.7]{N}. For the last two equalities, see also Propositions~\ref{10} and~\ref{11} below.
\end{proof}
Proposition~\ref{pr1} shows that the sequence $(b(n))_{n \geq0}$ is strictly related to the well studied 
Stern's diatomic sequence $(s(n))_{n \geq 0}$, i.e.\ the sequence A002487 of N.J.A. Sloane's Encyclopedia (see \cite{S}) also known as Stern-Brocot sequence. More precisely, we have
$$b(n)= s(n+1).$$
\begin{cor}\label{co1} The positive integer $n$ admits only one hyperbinary expansion if and only if $n=2^k-1$ for a suitable $k \in \N$.
\end{cor}
\begin{proof}
Note first that $b(1) = b (2^1-1)=1$. Suppose $n$ be equal to $2^k-1$ for some $k>1$. 
Since  $n=2(2^{k-1}-1)+1$, by Proposition~\ref{pr1} and induction on $k$ it follows that
$$ b(n) = b (2^{k-1}-1)=1.$$
Let now $n$ be equal to $2^k-1-j$ for some positive $j<2^k-2^{k-1}$, and let $\vec{n}''= x_1 \cdots x_k$ be the unique binary expansion of $n$. 
In this case,  the set  $\{ i \, \vert \;  x_i=0 \}$ is not empty, and we denote by $h$ its minimum. 
By definition, 
\begin{equation} \vec{n}'' = \underbrace{11 \cdots 1}_{h\text{-}2 \; \text{times}}\;  1 \; 0 \; x_{h+1} \cdots x_k.
\end{equation}
Such expansion has a parent  given by $2x_3\cdots x_k$ if $h=2$, and by  \begin{equation} \underbrace{11 \cdots 1}_{h\text{-}2 \; \text{times}}\;  0 \; 2 \; x_{h+1} \cdots x_k
\end{equation}
if instead $h >2$. In both cases the cardinality of $\mathcal H(2^k-1-j)$ is not less than $2$.
\end{proof}
For sake of completeness, we inform the reader that Corollary~\ref{co1}  could also be proved inductively in a shorter  (though less concrete) way from the striking equality
$$ b(2^k-1-j) = b(2^{k-1}-1+j) \quad \text{if $0 \leq j<2^k$} $$
noted by  S.~Northshield in \cite{N}.
\begin{prop}\label{pr2}
Each $n\in 	N$ admits a unique hyperbinary expansion without $0$'s. \end{prop}
\begin{proof} {\em (Existence)} Suppose we have found a hyperbinary expansion without $0$'s for all integers  $k \leq m$. We now describe how to obtain an expansion of required type for $m+1$. If $m$ is even and equal to $2k$, just add a $1$ on the right to the representation of $k$ without $0$'s.
If $m$ is odd, replace by $2$ the last $1$ on the right in the hyperbinary representation of $m$ without $0$'s.

{\em (Uniqueness)} Suppose the uniqueness of hyperbinary expansions without $0$'s proved up to the integer $m-1$, and let 
\begin{equation}\label{e1} x_1 \cdots x_{\ell} \quad \text{and} \quad y_1 \cdots y_{\ell'} \end{equation}
be two elements in $\mathcal H(m)$ without $0$'s.
Since $x_{\ell}$ and $y_{\ell'}$ depend on the parity of $m$, they are necessarily equal. It follows that
\begin{equation}\label{e2} x_1 \cdots x_{\ell -1} \quad \text{and} \quad y_1 \cdots y_{\ell' -1}
\end{equation} are both hyperbinary expansions without $0$'s of the integer $(m-x_{\ell})/2$. By the inductive hypothesis, the two words in \eqref{e2}, and hence the two words in \eqref{e1}, coincide.
\end{proof}

\begin{prop}\label{ac} The only element in $\mathcal H(n)$ with no ancestors with respect to
$\vec{R}_{\mathcal H}$ is the expansion $\vec{n}'$ without $0$'s. The only element in $\mathcal H(n)$ with no descendants is $\vec{n}''$, the binary expansion of $n$. 
\end{prop}
\begin{proof} By Proposition \ref{pr2}, the element $\vec{n}'$ is well defined, and by the rules \eqref{rew} it follows quite easily that an expansion has parents if and only if it has a zero. To operatively find a parent of an expansion $\vec{n}= x_1 \cdots  x_k$ not equal to $\vec{n}'$, we again set $h$  being the minimum of the set $\{ i \, \vert \;  x_i=0 \}$. Since $x_1$ is non-zero by definition, $h$ is at least $2$. The expansion $\vec{n}$ surely comes from a single-step reduction of type `$\rightarrow$' if $x_{h-1}=1$, and of type `$\twoheadrightarrow$' if $x_{h-1}=2$.

From the rewrite rules \eqref{rew} it also follows that the presence of a $2$ is a necessary and sufficient condition for a hyperbinary expansion to have children.
\end{proof}

In literature, a string $\vec{x}$ with no children with respect any set $\vec{S}$ of string-rewriting systems is often called $\vec{S}$-irreducible.

Assumed such notion, Proposition~\ref{ac} can be reworded by saying that  in $\mathcal H(n)$ there exists only one $\vec{R}_{\mathcal H}$-irreducible element: the binary expansion of $n$. In the same set, the expansion  $\vec{n}'$ without $0$'s is the unique $\vec{R}^{-1}_{\mathcal H}$-irreducible element, where $\vec{R}^{-1}_{\mathcal H}$ is the string-rewriting system induced  on 
$\mathcal H$  by 
$$ \vec{R}^{-1} :=  \{  (10, \,  02), \; (20, \, 12)  \} \subset \Sigma^* \times \Sigma^*.$$

As for any other set made by finite sequence of objects, elements in $\mathcal H = \cup_{n>0} \mathcal H(n)$ can be totally ordered by the {\em shortlex ordering} $<_{SL}$ (also known as radix or length-lexicographic ordering): sequences are primarily sorted by length with the shortest sequences first, and sequences of the same length are sorted into lexicographical order. The next Proposition shows that the  string-rewriting system $\vec{R}_{\mathcal H}$ is compatible with the shortlex ordering.
\begin{prop}\label{comp} For any $\vec{u}$ and $\vec{v}$ in  $\mathcal H(n)$ such that  $\vec{u}  \overset{*}{\rightarrow}_{\vec{R}_{\mathcal H}} \vec{v}$, we have
$  \vec{u}  <_{SL} \vec{v}$.
\end{prop}
\begin{proof}
It suffices to restrict our attention  to  single-step reductions. They can be of three different types:
\begin{equation}\label{rew2} I) \;\; 2 \, \vec{y} \rightarrow 1 \, 0 \, \vec{y},  \qquad  II) \;\; \vec{x} \,0 \, 2 \, \vec{y} \rightarrow \vec{x} \,1 \, 0 \, \vec{y},  \qquad \text{and} \qquad III) \;\; \vec{x} \,1 \, 2 \, \vec{y} \twoheadrightarrow \vec{x} \,2 \, 0 \, \vec{y},
\end{equation}
where $\vec{x}$ and $\vec{y}$ are suitable words  in $\mathcal H$. For reductions of type I) the parent is shorter than its child. In the other two cases parent and child have the same length, but the former precedes the latter with respect to the lexicographical order.
\end{proof}
Proposition~\ref{comp} has several noteworthy consequences. Some of them concern the so-called confluence (see \cite[p. 11]{BO} for the formal definition) of $\vec{R}_{\mathcal H}$ and $\vec{R}^{-1}_{\mathcal H}$ , and are framed in the next Corollary.
\begin{cor}\label{conf} Let $\vec{n}$ be any fixed element in  $\mathcal H(n)$.\\
i) All $\vec{R}_{\mathcal H}$-reductions starting with $\vec{n}$ end with the binary expansion $\vec{n}''$ of $n$.\\
ii) All $\vec{R}^{-1}_{\mathcal H}$-reductions starting with $\vec{n}$ end with the hyperbinary expansion $\vec{n}' \in \mathcal H(n)$ without $0$'s. \\
iii) The minimal and the maximal element in $\mathcal H(n)$ with respect to the shortlex ordering $<_{SL}$ are the expansions $\vec{n}'$ and $\vec{n}''$ respectively.
\end{cor}
\begin{proof} By Proposition~\ref{comp}, the strings $\vec{u}_1, \dots , \vec{u}_k \in \mathcal H (n)$ involved in the  reduction
$$ \vec{u}_1 \rightarrow_{\vec{R}_{\mathcal H}} \cdots \rightarrow_{\vec{R}_{\mathcal H}}  \vec{u}_k $$
are all distinct. It follows that $\vec{u}_1$ necessarily reaches $\vec{n}''$ (the unique $\vec{R}_{\mathcal H}$-irreducible element by Proposition~\ref{ac})  through at most $b(n)-1$ single-step reductions. The argument to prove Part ii) is similar.  The proof of Part iii) also comes from Proposition~\ref{ac}, once you note that, by Proposition~\ref{comp}, the minimal element with respect to $<_{SL}$ has no ancestors, and the maximal one has no descendants.
\end{proof}

\section{Setting up the Graph}\label{s3}
We are now ready to define the oriented graph $A(n) = (\mathcal H(n), \mathcal E(n))$ announced in the Introduction. The set of arcs $\mathcal E(n)$ is precisely given by the set of $\vec{R}_{\mathcal H}$-single-step reductions \eqref{rew2} between hyperbinary expansions in $\mathcal H(n)$, i.e.\ there is an arc from the node $\vec{x} \in \mathcal H(n)$ to $\vec{y}$ if and only if $\vec{y}$ is a child of $\vec{x}$ with respect to $\vec{R}_{\mathcal H}$. The definition of $\mathcal E(n)$ is well-posed by Proposition~\ref{p4}.
  
To each arc in  $\mathcal E(n)$
we assign one of the two following {\em colors}: color `$\rightarrow$'  to $\vec{R}_{\mathcal H}$-single-step reductions in \eqref{rew2} of type I) or II), and color `$\twoheadrightarrow$' to $\vec{R}_{\mathcal H}$-single-step reductions in \eqref{rew2} of type III.

We list below some properties of $A(n)$ implied by definition or by results collected in Section~\ref{s2}. 
\begin{itemize} 
\item The graph $A(n)$ is simple, i.e.\ it has no loops (no node is a parent of itself) or multiple edges (there is at most one arc connecting two nodes). 
\item $\mathcal E(n)$ is empty if and only if $n=2^k-1$ for a suitable $k \in \N$ (see Corollary~\ref{co1}).
\item $A(n)$ has a single root (i.e.\ a node with no ancestors) given by  the minimal hyperbinary expansion $\vec{n}'$ (see Proposition~\ref{ac}).
\item No node in $A(n)$ is an ancestor of itself (see Proposition~\ref{comp}).
\item $A(n)$ is a {\em flowchart}, i.e.\ is {\em pointed accessible} in the sense of \cite[p. 4]{A}: for every node $\vec{n} \not= \vec{n'}$ there exists at least one path from the root to  $\vec{n}$  (see Corollary~\ref{conf}, Part ii) ). 
\item The flowchart $A(n)$ has a {\em global sink}: all paths end to the binary expansion $\vec{n}''$, the unique node without children (see Proposition~\ref{ac} and Corollary~\ref{conf}, Part i) ). 
\end{itemize}
The alphabet $\Sigma$ used to compose words in $\mathcal H$ is a subset of $\N_0$. The map $$ \omega :  x_1 \cdots  x_k \in \mathcal H \longmapsto x_1 + \cdots +   x_k \in \N.$$
is therefore well defined and allows us to weight each node of $A(n)$.

A quick glimpse to $\vec{R}_{\mathcal H}$-single-step reductions \eqref{rew2} is all you need to prove the following Lemma.
\begin{lemma}\label{l1} Let $\vec{v} \in \mathcal H(n)$ be a child of $\vec{u}$. Then 
$ \omega (\vec{v} ) =   \omega (\vec{u} ) - 1.$
\end{lemma}
By Lemma~\ref{l1} and Corollary~\ref{conf} we immediately get the following Proposition.
\begin{prop} The restriction of the map $\omega$ to $\mathcal H (n)$ has a global minimum point and a global maximum point given by the sink $\vec{n}''$ and the root $\vec{n}'$ respectively.
\end{prop} 
 Properties of the graph $A(n)$ gathered so far suggest that the several nodes of $A(n)$ can be displayed like the vertices of a family tree. On the bottom line we put $\vec{n}'$, the common ancestor of all nodes. $\ell$ rows above we display all nodes $\vec{n}$  such that 
 $$ \omega (\vec{n})  =  \omega (\vec{n}') -\ell.$$ 
 All children of a fixed node $\vec{n}$ will be arrayed consecutively from the left to the right in decreasing shortlex ordering. The picture below visualizes the graphs $A(10)$ and $A(12)$. 
 For $n \not= 2^m-1$, we can use green, red and yellow to colour, respectively, the minimal hyperbinary expansion, the binary hyperbinary expansion, and the remaining {\em branching} nodes (if existing).

\bigskip
 
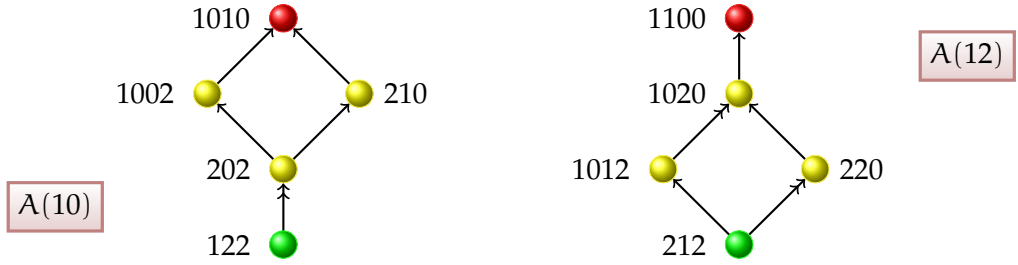
\begin{center}
\begin{tikzpicture}[
vertex_style/.style={circle,shading=ball,ball color=red,draw=red!80!white,
},
vertex2_style/.style={circle,shading=ball,ball color=green,draw=green!80!white,
},
vertex3_style/.style={circle,shading=ball,ball color=yellow,draw=yellow!80!white,drop shadow={opacity=0.0}},
edge_style/.style={ultra thick, black, drop shadow={opacity=0.1}},
edge_style/.style={ultra thick, black, drop shadow={opacity=0.1}},  
nonterminal/.style={
rectangle,
minimum size=6mm,
very thick,
draw=red!50!black!50, 
top color=white, 
bottom color=red!50!black!20, 
font=\itshape
} ]
 \node[vertex2_style] (A)  at (0,0)  [label=left:$122 \;$]{} ;
     \node[vertex3_style] (B) at (0,1)  [label=left:$202 \;$] {};
     \node[vertex3_style] (C) at (-1,2)  [label=left:$1002 \;$] {};
         \node[vertex3_style] (D) at (1,2)  [label=right:$210$] {};
          \node[vertex_style] (E)  at (0,3)  [label=left:$1010 \;$]{} ;
        \path (A) edge [->>, thick]  (B);
           \path [->, thick] (B) edge (C);     
            \path [->, thick] (B) edge (D);    
          \path [->, thick] (C) edge (E); 
        \path  (D) edge [->, thick]  (E);
            \node[nonterminal] (G) at (-3,0.5) {$A(10)$};
\end{tikzpicture}
\qquad \qquad
\begin{tikzpicture}[
vertex_style/.style={circle,shading=ball,ball color=red,draw=red!80!white,
},
vertex2_style/.style={circle,shading=ball,ball color=green,draw=green!80!white,
}, vertex3_style/.style={circle,shading=ball,ball color=yellow,draw=yellow!80!white,drop shadow={opacity=0.0}
},
nonterminal/.style={
rectangle,
minimum size=6mm,
very thick,
draw=red!50!black!50, 
top color=white, 
bottom color=red!50!black!20, 
font=\itshape
}]
 \node[vertex2_style] (A)  at (0,0)  [label=left:$212 \;$]{} ;
       \node[vertex3_style] (C) at (-1,1)  [label=left:$1012 \;$] {};
         \node[vertex3_style] (D) at (1,1)  [label=right:$220 \,$] {};
          \node[vertex3_style] (E)  at (0,2)  [label=left:$1020 \;$]{} ;
             \node[vertex_style] (F) at (0,3)  [label=left:$1100 \;$] {};
           \path [->, thick] (A) edge (C);     
            \path [->>, thick] (A) edge (D);    
          \path [->>, thick] (C) edge (E); 
        \path  (D) edge [->, thick]  (E);
          \path  (E) edge [->, thick]  (F);
          \node[nonterminal] (G) at (3,2.5) {$A(12)$};
\end{tikzpicture}
\end{center}
\bigskip
\begin{center}
\begin{tikzpicture}
[nonterminal/.style={
rectangle,
minimum size=6mm,
very thick,
draw=red!50!black!50, 
top color=white, 
bottom color=red!50!black!20, 
font=\itshape
}]
\end{tikzpicture}
\end{center}

Thus, all nodes of $A(n)$ are disposed along $\omega(\vec{n}') - \omega(\vec{n}'') +1$ rows.
Furthermore, the presence of a fixed node $\vec{n} \in \mathcal H (n)$ on a certain row measures how far $\vec{n}$ is from being binary or minimal hyperbinary. In fact,
after setting $$ i(\vec{n}) = \omega (\vec{n}) - \omega (\vec{n}''), \quad \text{and} \quad j(\vec{n}) = \omega (\vec{n}') - \omega (\vec{n}),$$
the expansion $\vec{n}$ is  $i (\vec{n})$ generations away from being binary, since, by Lemma~\ref{l1}, the number $i (\vec{n})$ counts the $\vec{R}_{\mathcal H(n)}$-single-step reductions needed to reduce $\vec{n}$  to the bynary expansion. Analogously,  $\vec{n}$ is $j (\vec{n})$ generations away from being minimal hyperbinary.

Such criteria to establish a kind of distance from {\em binarity} and from  {\em minimal-hyperbinarity} are finer than simply counting the number of $2$'s in $\vec{n}$, and more mathematically significant than the shortlex ordering:  the expansions $122$ and $202$ in $\mathcal H(10)$ share the same number of $2$'s, but 
$$3 = i(122) \not= i(202) = 2  \quad \text{and} \quad 0= j(122)\not=j(202)=1;$$
Moreover, in $\mathcal H(20)$ we have
$$ 1212 <_{SL} 2100 <_{SL} 10012 <_{SL} 10100,$$
yet the expansion $2100$ is closer to be binary than $10012$ since 
$$1= i(2100) < i(10012)=2.$$

 We now state two Propositions that, in some way, geometrically translate the statement of  Proposition~\ref{pr1}.  

\newpage

\begin{prop}\label{10} For each $m \in \N$, there is a oriented graph isomorphism between $A(m)$ and $A(2m+1)$.
\end{prop}
\begin{proof}
The map
$$ \psi :  \vec{m} = x_1 \cdots x_k \in \mathcal H (m) \longmapsto x_1 \cdots x_k \, 1 \in \mathcal H (2m+1)$$
is a bijection, and preserves relations: if $\vec{y}$ is a child of $\vec{x}$, then $\psi (\vec{y})$ is a child of $\psi (\vec{x})$ through a single-step relation of the same type. The proof ends once you note that the word $x_k \, 1$ has no descendants; hence the expansions $ \vec{m} $ and $\psi (\vec{m})$ have the same number of children.
\end{proof}
 The following Proposition involves the notion of a {\em graph monomorphism} consisting, from our point of view, of an injective vertex map respecting connectivity between vertices and colors of the arcs.
 \begin{prop}\label{11}  The maps  $$ \phi' : x_1 \cdots x_k \in \mathcal H (m) \longmapsto x_1 \cdots x_k \, 0 \in \mathcal H (2m) $$
 and 
  $$ \phi'' : x_1 \cdots x_k \in \mathcal H (m-1) \longmapsto x_1 \cdots x_k \, 2 \in \mathcal H (2m)$$
  induce two graph monomorphisms from $A(m)$ and from $A(m-1)$ respectively to $A(2m)$. Moreover
 \begin{equation}\label{eq10} \phi' (H (m)) \, \cup \, \phi'' (\mathcal H (m-1))  = \mathcal H(2m).
 \end{equation}
\end{prop}
\begin{proof} The proof is straightforward. Equation \eqref{eq10} comes from the fact that, being the integer $2m$ even, no hyperbinary expansions of $2m$ end with $1$.
\end{proof}
Proposition~\ref{11} says that each graph $A(2m)$ is essentially made by two pieces:  we find a subgraph isomorphic to $A(m-1)$  standing somewhat on the left; its nodes are given by the set $\phi'' (\mathcal H (m-1))$. The remaining nodes in $A(2m)$ and the arcs connecting them form a subgraph isomorphic to  $A(m)$. The situation is visualized in the picture below, where the red-shaded subgraph of $A(18)$ is  isomorphic to $A(8)$, and the blue-shaded one is isomorphic to $A(9)$. 

\medskip

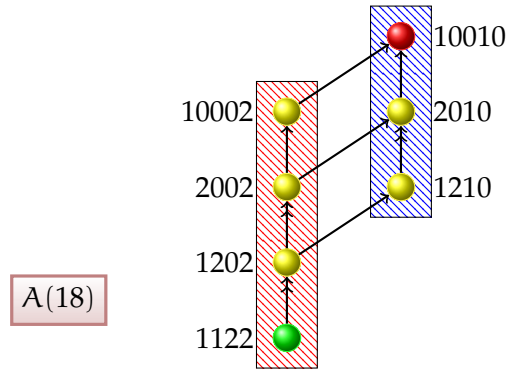
\begin{center}
\begin{tikzpicture}[
vertex_style/.style={circle,shading=ball,ball color=red,draw=red!80!white,
},
vertex2_style/.style={circle,shading=ball,ball color=green,draw=green!80!white,
},
vertex3_style/.style={circle,shading=ball,ball color=yellow,draw=yellow!80!white,drop shadow={opacity=0.0}},
edge_style/.style={ultra thick, black, drop shadow={opacity=0.1}},
edge_style/.style={ultra thick, black, drop shadow={opacity=0.1}},  
nonterminal/.style={
rectangle,
minimum size=6mm,
very thick,
draw=red!50!black!50, 
top color=white, 
bottom color=red!50!black!20, 
font=\itshape
} ]
   \draw [pattern= north west lines, pattern color=red] (-0.4,-0.4)--(-0.4,3.4) to (0.4,3.4) -- (0.4,-0.4) --cycle;
              \draw [pattern= north west lines, pattern color=blue] (1.1,1.6)--(1.1,4.4) to (1.9,4.4) -- (1.9,1.6) --cycle;
 \node[vertex2_style] (A)  at (0,0)  [label=left:$1122 \;$]{} ;
     \node[vertex3_style] (B) at (0,1)  [label=left:$1202 \;$] {};
     \node[vertex3_style] (C) at (0,2)  [label=left:$2002 \;$] {};
         \node[vertex3_style] (D) at (1.5,2)  [label=right:$\; 1210$] {};
          \node[vertex3_style] (E)  at (0,3)  [label=left:$10002 \;$]{} ;
               \node[vertex3_style] (G)  at (1.5,3)  [label=right:$\; 2010 \;$]{} ;
          \node[vertex_style] (F)  at (1.5,4)  [label=right:$\; 10010 \;$]{} ;
        \path (A) edge [->>, thick]  (B);
           \path [->>, thick] (B) edge (C);     
            \path [->, thick] (B) edge (D);    
          \path [->, thick] (C) edge (E); 
             \path [->, thick] (E) edge (F); 
              \path [->, thick] (C) edge (G); 
        \path  (D) edge [->>, thick]  (G);
                  \path [->, thick] (G) edge (F); 
            \node[nonterminal] (H) at (-3,0.5) {$A(18)$};
          
\end{tikzpicture}
\end{center}

\medskip
By looking at the picture above, the reader has probably noted that all arcs in $\mathcal E(18)$ from the red-shaded subgraph to the blue-shaded one are of the same type. This is actually a general feature of all graphs $A(2m)$ with $ m \in \N$, in a sense made precise by Proposition~\ref{cr}.  Its statement concerns {\em bridging arcs}, i.e.\ edges in $A(2m)$ from a node in   $\phi' (\mathcal H (m-1))$ to a node
in $\phi'' (\mathcal H (m))$.

\begin{prop}\label{cr} Fixed any integer $m \geq 2$, all bridging arcs in $\mathcal E(2m)$ are equally colored. They are of type `$\rightarrow$' if $m$ is odd, and of type `$\twoheadrightarrow$' otherwise.
\end{prop}
\begin{proof}  A node $  x_1 \cdots x_{k-1} \, 2 \in \phi'' (\mathcal H (m-1))$ has a child in $ \phi'' (\mathcal H (m))$ if and only if $x_{k-1} \not=2$. In this case, the color of the bridging arc connecting them is `$\rightarrow$' if and only if $x_{k-1}=0$. This happens whenever $m-1$ is even.
\end{proof}
As a final remark on the general properties of the sequence of graphs $\{ \,  A(n) \, \vert \, n >0 \, \}$, we point out that its elements are not necessarily planar, the minimal integer such that $A(n)$ is non-planar being $n=36$.

\section{Long vs. Short Expansions}\label{s4}

A hyperbinary expansion $x_ 1 \cdots x_k$ of $n$  is said to be {\em short} or {\em long}  whether $k = \lfloor \log_2 n \rfloor$ or $\lfloor \log_2 n \rfloor + 1$. In \cite[Section 3]{AR}, A. De Luca and C. Reutenauer explain why these are the only possible lengths. 

\begin{lemma} \label{l1} The hyperbinary expansions of $n$ of type $2 \, x_2\cdots x_h$ and $1 \, 2 \, x_3 \cdots x_{h}$ are necessarily short. In other words, $h$ is necessarily equal to $\lfloor \log_2 n \rfloor$.
\end{lemma}
\begin{proof}
In the graph $A(n)$, the node $2 \, x_2\cdots x_h$  has $1\, 0 \,  x_2\cdots x_h$ among its children. Since the latter has a longer length, the former is short. Analogously, the presence of the node $1 \, 2 \, x_3 \cdots x_h$ implies the existence in $A(n)$ of the subgraph
\begin{equation}\label{aa} 1 \, 2 \, x_3 \cdots x_h \;  \twoheadrightarrow \; 2 \, 0 \, x_3 \cdots x_h \; \rightarrow \; 1 \, 0 \, 0 \, x_3 \cdots x_h
\end{equation}
isomorphic to $A(4)$. The third expansion in \eqref{aa}  is longer of its (necessarily short) ancestors.
\end{proof}

\begin{prop}\label{p1} Let $n$ be any positive integer not equal to a power of $2$. The long hyperbinary expansions, together with those arcs in $\mathcal E(n)$ connecting them, form a subgraph $L (n)$ of $A(n)$ isomorphic to $A(n-2^k)$, where 
$ k= \lfloor \log_2 n \rfloor $.
\end{prop}
\begin{proof} 
By Lemma \ref{l1}, a typical long hyperbinary expansion of $n$ is either of type $1 \, 0 \, x_3 \cdots x_{k+1}$ or  $1 \, 1 \, x_3 \cdots x_{k+1}$, where $k= \lfloor \log_2 n \rfloor$.
The one-to-one map between the nodes of $L (n)$ and those in $A(n-2^k)$ is given by
$$\xi : 1 \, x_2 \, x_3 \cdots x_{k+1}  \longmapsto x_{\bar{h}} \cdots x_{k+1},$$ where
$ \bar{h} = \min \; \{ \; h \geq 2 \; \vert \; x_h \not=0 \; \}.$
Such $\bar{h}$ exists since, in our hypotheses, $n$ is not a power of $2$.

By construction, whenever the nodes $\vec{n}_1$ and $\vec{n}_2$ in $L(n)$ are connected by an arc, $\xi (\vec{n}_1)$ and $\xi (\vec{n}_2)$ in $A(n-2^k)$ are connected by an arc of the same color. 
 \end{proof}
As an example, consider the case $n=20$.
Since $ 2^4 < 20 <2^5$, according to Proposition \ref{p1} the subgraph of  $A(20)$ of long expansions is isomorphic to $A(4)$, as visualized below.
\bigskip
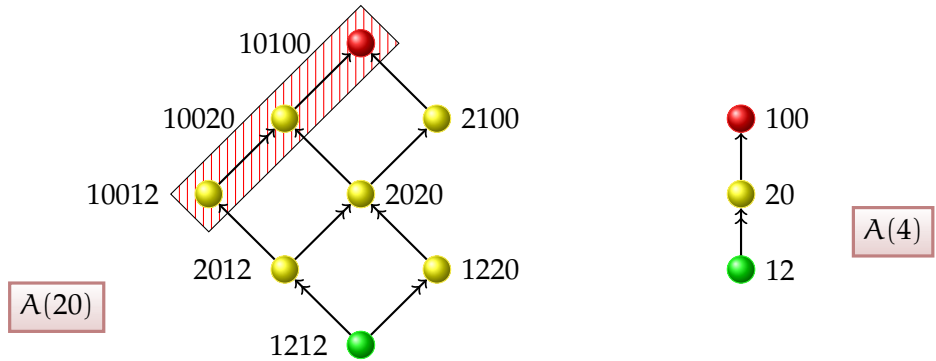
\begin{center}
\begin{tikzpicture}[
vertex_style/.style={circle,shading=ball,ball color=red,draw=red!80!white,
},
vertex2_style/.style={circle,shading=ball,ball color=green,draw=green!80!white,
}, vertex3_style/.style={circle,shading=ball,ball color=yellow,draw=yellow!80!white,drop shadow={opacity=0.0}
},
nonterminal/.style={
rectangle,
minimum size=6mm,
very thick,
draw=red!50!black!50, top color=white, 
bottom color=red!50!black!20, font=\itshape
}]
   \draw [pattern= vertical lines, pattern color=red] (-2,1.5)--(-2.5,2) to (0,4.5) -- (.5,4) --cycle;
 \node[vertex2_style] (A)  at (0,0)  [label=left:$1212 \;$]{} ;
       \node[vertex3_style] (C) at (-1,1)  [label=left:$2012 \;$] {};
         \node[vertex3_style] (D) at (1,1)  [label=right:$1220 \,$] {};
          \node[vertex3_style] (E)  at (0,2)  [label=right:$2020 \;$]{} ;
          \node[vertex3_style] (G)  at (-2,2)  [label=left:$10012 \;\;\;$]{} ;
           \node[vertex3_style] (H) at (-1,3)  [label=left:$10020 \;\;\;$] {};
              \node[vertex3_style] (I) at (1,3)  [label=right:$2100 \;$] {};
             \node[vertex_style] (F) at (0,4)  [label=left:$10100 \;\;\;$] {};
           \path [->>, thick] (A) edge (C);     
            \path [->>, thick] (A) edge (D);    
          \path [->>, thick] (C) edge (E); 
        \path  (D) edge [->>, thick]  (E);
              \path  (G) edge [->>, thick]  (H);
          \path  (H) edge [->, thick]  (F);
          \path  (C) edge [->, thick]  (G);
            \path  (E) edge [->, thick]  (H);
             \path  (E) edge [->, thick]  (I);
              \path  (I) edge [->, thick]  (F);
          \node[nonterminal] (Z) at (-4,.5) {$A(20)$};
             \node[vertex2_style] (S)  at (5,1)  [label=right:$12 \;$]{} ;
       \node[vertex3_style] (T) at (5,2)  [label=right:$20 \;$] {};
             \node[vertex_style] (U) at (5,3)  [label=right:$100 \;$] {};
           \path [->>, thick] (S) edge (T);     
            \path [->, thick] (T) edge (U);   
           
          \node[nonterminal] (V) at (7,1.5) {$A(4)$};
\end{tikzpicture}

\end{center}

\begin{cor}\label{c1} Let $n$ be any positive integer, and $ k= \lfloor \log_2 n \rfloor $.
The number $ \ell (n) $ of long hyperbinary expansions of $n$ is given by  
$ b(n-2^k)$.
\end{cor}
\begin{proof}
The only case left aside by Proposition \ref{p1} is when $n=2^k$ for some $k \in \N_0$. 
We have to prove that each $n=2^k$ has just $b(0)=1$ long expansion. 
This is trivially true for $k \leq 1$. 

When $k >1$, we note that the nodes of the tree $A(2^k)$ are
$$ 1\underbrace{0 \dots 0}_{\text{$k$ times}}\; ; \qquad 2\underbrace{0 \dots 0}_{\text{$k \text{-} 1$ times}} \, ; \qquad \underbrace{1 \dots 1}_{\text{$h$ times}} \; 2 \! \underbrace{0 \dots 0}_{\text{$k \text{-} h \text{-}1$ times}}  \quad \text{for $0<  h < k$.}$$
Among such expansions only the first one is long.
\end{proof}

When $n=2^k-1$, Corollary \ref{c1} says that the number of long expansions is equal to
$$ b(2^k-1-2^{k-1}) =  b(2^{k-1}-1) = 1. $$
In fact the unique hyperbinary expansion of  $2^k-1$ is long.

\section{Trees}\label{s5}
The fundamental group $\pi_1(G)$ of any connected graph $G$ is a free group.  The number $v(G)$ of its generators is called {\em cyclomatic number} or  {\em circuit rank} of $G$. 
\begin{prop}\label{15} Let $n$ be a positive integer not equal to $2^m-1$. The following formula holds.
$$v(A(n))= \sum_{\vec{n} \not=\vec{n}''} ( o(\vec{n})-1 ),$$
where $o(\vec{n})$ is the outdegree of $\vec{n}$, i.e.\ the number of its children.
\end{prop}
\begin{proof}
It is well known (see, for instance, \cite[p. 67]{O}) that
$$ v(A(n)) = \lvert \mathcal E (n) \rvert  - ( \lvert \mathcal H (n) \rvert -1).$$
The cardinality of  $\mathcal E (n)$ is given by 
$ \sum_{\vec{n} \not=\vec{n}''}  o(\vec{n})$, since $\vec{n}''$ has no children.
The result follows now easily.
\end{proof}
\begin{lemma}\label{16} The outdegree of a node $\vec{n} = x_1 \cdots x_k \in \mathcal H(n)$ is given by  the number of its blocks of $2$'s.
\end{lemma}
\begin{proof} To simplify notation, we represent $\vec{n}$ by the word $0 \, x_1 \cdots x_k \in \Sigma^*$, and set $x_0=0$. Suppose $\vec{n}$ contains $\ell$ blocks of $2$'s, and let $x_{i_h}$ be the first $2$ of the $h$-th block. Since the elements $x_{i_1-1}, \dots , x_{i_{\ell}-1}$ are all in the set $\{0,1\}$, there are precisely $\ell$ different $\vec{R}$-single-step reductions operating on $\vec{n}$, the $h$-th of them transforming the word
$ x_{i_h-1} \, 2$  into $y_{i_h-1} \, 0$, where $y_{i_h-1} = x_{i_h-1}+1$.
 \end{proof}
 \begin{cor}\label{17} Let $\ell$ be the number of blocks of $2$'s in the minimal hyperbinary expansion $\vec{n}' \in \mathcal H(n)$. We have $v(A(n)) \geq \ell -1$.
 \end{cor} 
 \begin{proof} Immediate from Proposition~\ref{15} and Lemma~\ref{16}.
 \end{proof}
We are now ready to identify those integers $n \in \N$ such that $ v(A(n))=0$.
\begin{thm}\label{t1} The graph $A(2m)$ is a tree if and only if $m= 2^t - \epsilon$ for a suitable $(t, \epsilon) \in \N \times \{0,1\}$.
\end{thm}
\begin{proof} By Corollary~\ref{17}, the cyclomatic number $v(A(2m))$ is possibly $0$ only if the minimal hyperbinary expansion
$\vec{n}' \in \mathcal H(2m)$ is equal to
$$ \underbrace{1 \dots 1}_{\text{$h$ times}} \; \underbrace{2 \dots 2}_{\text{$k \text{-} h$ times}}  \quad \text{for $0\leq   h < k$.}$$
When $h=0$, we get $m=2^k-1$. The binary expansion of $2m$ is  $ \underbrace{1 \dots 1}_{\text{$k$ times}} \, 0$, and the graph $A(2m)$ is actually a tree with $k+1$ nodes.

Suppose now $h>0$. If $k-h=1$, we get $m=2^h$, and $A(2m)$ is again a tree with $k+1$ nodes.
Suppose finally $h>0$ and $k-h>1$. In this case
$\vec{n}'= \vec{x} \, 1 \, 2 \, 2 \, \vec{y}$ for suitable  $\vec{x}$ and $\vec{y}$ in $\Sigma^*$. This implies that
 $A(2m)$ contains a subgraph isomorphic to either $G(1)$ or $G(2)$ below
 \medskip

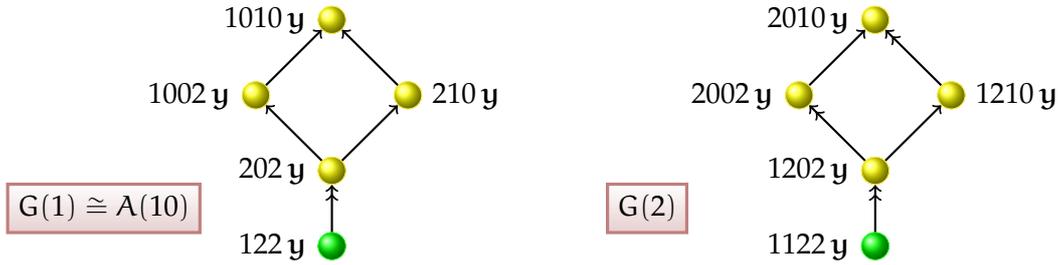
\begin{center}
\begin{tikzpicture}[
vertex_style/.style={circle,shading=ball,ball color=red,draw=red!80!white,
},
vertex2_style/.style={circle,shading=ball,ball color=green,draw=green!80!white,
},
vertex3_style/.style={circle,shading=ball,ball color=yellow,draw=yellow!80!white,drop shadow={opacity=0.0}},
edge_style/.style={ultra thick, black, drop shadow={opacity=0.1}},
edge_style/.style={ultra thick, black, drop shadow={opacity=0.1}},  
nonterminal/.style={
rectangle,
minimum size=6mm,
very thick,
draw=red!50!black!50, 
top color=white, 
bottom color=red!50!black!20, 
font=\itshape
} ]
 \node[vertex2_style] (A)  at (0,0)  [label=left:$122 \,\vec{y}$]{} ;
     \node[vertex3_style] (B) at (0,1)  [label=left:$202 \, \vec{y}$] {};
     \node[vertex3_style] (C) at (-1,2)  [label=left:$1002 \, \vec{y}$] {};
         \node[vertex3_style] (D) at (1,2)  [label=right:$210\, \vec{y}$] {};
          \node[vertex3_style] (E)  at (0,3)  [label=left:$1010 \, \vec{y}$]{} ;
        \path (A) edge [->>, thick]  (B);
           \path [->, thick] (B) edge (C);     
            \path [->, thick] (B) edge (D);    
          \path [->, thick] (C) edge (E); 
        \path  (D) edge [->, thick]  (E);
             \node[nonterminal] (G) at (-3,0.5) {$G(1) \cong A(10)$};
\end{tikzpicture}
\qquad \quad 
\begin{tikzpicture}[
vertex_style/.style={circle,shading=ball,ball color=red,draw=red!80!white,
},
vertex2_style/.style={circle,shading=ball,ball color=green,draw=green!80!white,
},
vertex3_style/.style={circle,shading=ball,ball color=yellow,draw=yellow!80!white,drop shadow={opacity=0.0}},
edge_style/.style={ultra thick, black, drop shadow={opacity=0.1}},
edge_style/.style={ultra thick, black, drop shadow={opacity=0.1}},  
nonterminal/.style={
rectangle,
minimum size=6mm,
very thick,
draw=red!50!black!50, 
top color=white, 
bottom color=red!50!black!20, 
font=\itshape
} ]
 \node[vertex2_style] (A)  at (0,0)  [label=left:$1122 \, \vec{y}$]{} ;
     \node[vertex3_style] (B) at (0,1)  [label=left:$1202 \, \vec{y}$] {};
     \node[vertex3_style] (C) at (-1,2)  [label=left:$2002 \, \vec{y}$] {};
         \node[vertex3_style] (D) at (1,2)  [label=right:$1210 \, \vec{y}$] {};
          \node[vertex3_style] (E)  at (0,3)  [label=left:$2010 \, \vec{y}$]{} ;
        \path (A) edge [->>, thick]  (B);
           \path [->>, thick] (B) edge (C);     
            \path [->, thick] (B) edge (D);    
          \path [->, thick] (C) edge (E); 
        \path  (D) edge [->>, thick]  (E);
             \node[nonterminal] (G) at (-3,0.5) {$G(2)$};
\end{tikzpicture}\end{center}
\bigskip
\begin{center}
\begin{tikzpicture}
[nonterminal/.style={
rectangle,
minimum size=6mm,
very thick,
draw=red!50!black!50, 
top color=white, 
bottom color=red!50!black!20, 
font=\itshape
}]
\end{tikzpicture}
\end{center}
\noindent
depending whether $\vec{x}$ is empty or not.
In both occurences, $1 = v(G(i)) \leq v(A(m))$, hence $A(m)$ is not a tree. 
\end{proof}
\begin{thm}\label{t2} The graph $A(n)$ is a tree if and only if there exists $(s,t) \in \N_0 \times \N_0$ such that
\begin{equation}\label{e12} n = 2^{s+t+1} \Mypm 2^s-1>0. \end{equation}
\end{thm}
\begin{proof} Suppose first $n$ be even. By Theorem~\ref{t1}, the graph $A(n)$ is a tree if and only if
$n=2(2^t-\epsilon)$ for a suitable $(t, \epsilon) \in \N \times \{0,1\}$, and these are precisely the numbers 
appearing in Equation~\eqref{e12} when $s=0$.

Suppose now $n$ be odd. By Corollary~\ref{17}, the graph $A(n)$ is possibly a tree only if the number of blocks of $2$'s in the minimal hyperbinary expansion $\vec{n}'$ of $n$ is at most $1$. If there are no $2$'s in $\vec{n}'$, then $n =2^s-1$ for a suitable $s>0$, and $A(n)$ is actually a tree (with a single node and no arcs).  Such class of integers is obtained in Equation~\eqref{e12} by setting $t=0$ and read `$\Mypm$' as `$-$'. 

If instead
$$ \vec{n}' = x_1 \cdots x_k \, 2 \, \underbrace{1 \dots 1}_{\text{$s$ times}} \, ,$$
then the graph $A(n)$ is isomorphic to $A(m)$ with
$$m=  \frac{n - 2^{s}+1}{2^s},$$
as a consequence of Proposition~\ref{10}.

Note that $m$ is even, since its minimal hyperbinary expansion end with a $2$. 
By Theorem~\ref{t1} the graph $A(m)$ is a tree if and only if $m= 2(2^t - \epsilon)$ for a suitable $(t, \epsilon) \in \N \times \{0,1\}$. Equation~\eqref{e12} (with $t>0$) now comes by solving
\begin{equation}\label{e13}  2(2^t - \epsilon) = \frac{n - 2^{s}+1}{2^s} \end{equation}
with respect to $n$.
\end{proof}
\begin{remark} Accomplished the proof of Theorem~\ref{t2}, the careful reader may have some doubts whether the graphs $A(2^{s+1}+2^s-1)$ for $s>0$ are actually trees. The answer, coherently with the statement of Theorem~\ref{t2}, is positive.
It can be verified directly, or by recognizing that
the integers at hand come out from Equation~\eqref{e13} when $t=1$ and $\epsilon=1$; if fact they can also be written  as $2^{s+2}-2^s-1$.
\end{remark}

 Theorems~\ref{t2} holds to a nice  formula involving both the function $b$ and the map $\omega$.
 \begin{cor} Let $n$ be any positive integer. Denoted by $\vec{n}'$ and $\vec{n}''$ the root and the sink of $A(n)$, the following inequality holds.
 \begin{equation}\label{e14} b(n) \geq \omega (\vec{n}') -   \omega (\vec{n}'')+1.
 \end{equation}
 The equality in \eqref{e14} holds if and only if $n = 2^{s+t+1} \Mypm 2^s-1$ for a suitable $(s,t) \in \N_0 \times \N_0$.
 \end{cor}
 \begin{proof} Remember that the number  $\omega (\vec{n}') -   \omega (\vec{n}'')+1$ counts the rows along which the $b(n)$ nodes of $A(n)$ are arrayed. Moreover, there is a single node on each row if and only if $A(n)$ is a tree. Now apply Theorem~\ref{t2}.
 \end{proof}
\section{When $\pi_1(A(n))$ is $\Z$}\label{s6}
Let $S_1$ be the set of integers $n>0$ such that $v(A(n))=1$.
The minimal $n$ in $S_1$ is $10$, followed by $12$. Proposition~\ref{10} implies that 
\begin{equation}\label{15} T= \{ \; 2^{\ell} (12 \Mypm 1) -1 \; \vert \; \ell \geq 0 \; \} \subseteq S_1.
\end{equation}
We intend to show that the set $T$ is actually equal to $S_1$.

Fixed any $n \in S_1$, 
 we know by Corollary~\ref{17} that the number of blocks of $2$'s in the  root $\vec{n}'$ of $A(n)$ is either $1$ or $2$. If the former is the case, we can assume that $ \vec{n}' $ is represented by the word 
$$  \underbrace{1 \dots 1}_{\text{$h$ times}} \;  \underbrace{2 \dots 2}_{\text{$k$ times}} \;  \underbrace{1 \dots 1}_{\text{$\ell$ times}},$$
where $h>0$ and $k >1$ (otherwise $A(n)$ would be a tree). 
It follows that $A(n)$ contains a subgraph $G(3)$ rooted in $$ \underbrace{1 \dots 1}_{\text{$h$ times}} \;  2 \; 2 \, , $$ and hence isomorphic to $A(2^{h+2}+2)$. By a direct analysis we get 
$$h=v(G(3)) \leq v(A(n))=1.$$ 

The cyclomatic number of the graph $A(n)$ when
$$ \vec{n}' = 1 \,  \underbrace{2 \dots 2}_{\text{$k$ times}} \;  \underbrace{1 \dots 1}_{\text{$\ell$ times}}$$
can be computed by hand, and turns out to be $k-1$. 
 Thus we can infer that $k=2$.  Hence
$$ n = 2^{\ell} (10+1) -1 = 2^{\ell} (12 - 1) -1 \in T$$
as we claimed. 

Assume now $\vec{n}'$ having two blocks of $2$'s, separated by a block of $1$'s of length $k>0$. 
By Lemma~\ref{16}, the expansion $\vec{n}'$ has two children. All the others branching nodes in $A(n)$ have consequently just one child by Proposition~\ref{15}, and hence just one block of $2$'s.
This is only possible if 
 $$ \vec{n}' = 2 \, 1 \, 2 \,  \underbrace{1 \dots 1}_{\text{$\ell$ times}}\, ,$$
 i.e.\ $n = 2^{\ell} (12 + 1) -1 \in T.$

 By the equality of sets $T= S_1$ we also deduce that all graphs $A(n)$ such that $\pi_1(A(n))=\Z$ are either isomorphic to $A(10)$ or $A(12)$.

\end{document}